\documentclass{amsart}
\usepackage{amsmath,amssymb,amsthm}

\newtheorem{thm}{Theorem}[section]
\newtheorem{lem}[thm]{Lemma}

\theoremstyle{definition}

\newtheorem{rem}[thm]{Remark}

\numberwithin{equation}{section}

\newcommand{\N}{\mathbb{N}}
\newcommand{\Z}{\mathbb{Z}}
\newcommand{\Q}{\mathbb{Q}}

\newcommand{\C}{\mathbb{C}}
\newcommand{\bk}{\mathbf{k}}
\newcommand{\frH}{\mathfrak{H}}
\newcommand{\frs}{\mathfrak{s}}
\newcommand{\frt}{\mathfrak{t}}

\newcommand{\calZ}{\mathcal{Z}}
\DeclareMathOperator{\Ker}{\mathrm{Ker}}

\author{Shuji Yamamoto}
\address{
JSPS Research Fellow \\ 
Graduate School of Mathematical Sciences\\ 
The University of Tokyo\\ 
3-8-1 Komaba, Meguro, Tokyo, 153-8914 Japan.}
\email{yamashu@ms.u-tokyo.ac.jp}
\thanks{This work was supported by Grant-in-Aid for JSPS Fellows 21$\cdot$5093}


\title[Evaluation of sums of multiple zeta-star values]
{Explicit evaluation of certain sums of multiple zeta-star values}

\begin{document}
\maketitle 

\begin{abstract}
Bowman and Bradley proved an explicit formula 
for the sum of multiple zeta values whose indices are 
the sequence $(3,1,3,1,\ldots,3,1)$ with a number of $2$'s inserted. 
Kondo, Saito and Tanaka considered the similar sum of 
multiple zeta-star values and showed that this value is 
a rational multiple of a power of $\pi$. 
In this paper, we give an explicit formula for the rational part. 
In addition, we interpret the result as an identity in the harmonic algebra. 
\end{abstract}

\section{Introduction}
Let us consider the multiple zeta values (MZV, for short) 
\begin{equation*}
\zeta(k_1,\ldots,k_n)=\sum_{m_1>\cdots>m_n>0}
\frac{1}{m_1^{k_1}\cdots m_n^{k_n}}. 
\end{equation*}
In some cases, explicit evaluations are known 
for these values or sums of them. 
For example, there are the formulas 
\begin{align}
\zeta(\{2\}^q)&=\frac{\pi^{2q}}{(2q+1)!}, \label{eq:s(0,q)}\\
\zeta(\{3,1\}^p)&=\frac{\pi^{4p}}{(2p+1)(4p+1)!} \label{eq:s(p,0)}
\end{align}
(the notation $\{\ \ \}^p$ means that the sequence in the bracket 
is repeated $p$-times). 
In fact, these values are the special cases $s(0,q)$ and $s(p,0)$ of 
the following sums of MZVs 
\begin{equation*}
s(p,q)=\sum_{\substack{j_0,j_1,\ldots,j_{2p}\geq 0\\ j_0+j_1+\cdots+j_{2p}=q}}
\zeta(\{2\}^{j_0},3,\{2\}^{j_1},1,\{2\}^{j_2},3,\ldots,
3,\{2\}^{j_{2p-1}},1,\{2\}^{j_{2p}}), 
\end{equation*}
for which an explicit formula was given by Bowman-Bradley \cite{BB}: 
\begin{equation}\label{eq:s(p,q)}
s(p,q)={2p+q\choose q}\frac{\pi^{4p+2q}}{(2p+1)(4p+2q+1)!}. 
\end{equation}

On the other hand, we may also consider 
the multiple zeta-star values (MZSV for short) 
\begin{equation*}
\zeta^\star(k_1,\ldots,k_n)=\sum_{m_1\geq\cdots\geq m_n\geq 1}
\frac{1}{m_1^{k_1}\cdots m_n^{k_n}}. 
\end{equation*}
As an analogue of $s(p,q)$, we put 
\begin{equation*}
s^\star(p,q)
=\sum_{\substack{j_0,j_1,\ldots,j_{2p}\geq 0\\ j_0+j_1+\cdots+j_{2p}=q}}
\zeta^\star(\{2\}^{j_0},3,\{2\}^{j_1},1,\{2\}^{j_2},3,\ldots,
3,\{2\}^{j_{2p-1}},1,\{2\}^{j_{2p}}). 
\end{equation*}
Then the theorem of Kondo-Saito-Tanaka \cite{KST} states that 
$s^\star(p,q)\in\Q\pi^{4p+2q}$ (see also \cite{Tanaka}). 
The rational part, however, has not been given explicitly 
except for the cases $p=0$ (Zlobin \cite{Zlobin}) 
and $q=0,1$ (Muneta \cite{Muneta}). 
The formula for $p=0$ is 
\begin{equation}\label{eq:s^star(0,q)}
s^\star(0,q)=\zeta^\star(\{2\}^q)=
(2^{2q}-2)\frac{(-1)^{q-1}B_{2q}}{(2q)!}\pi^{2q} 
\end{equation}
($B_{2q}$ is the $2q$-th Bernoulli number). 

In this paper, we prove the following relation between 
$s(p,q)$ and $s^\star(p,q)$: 

\begin{thm}\label{thm:main 3,1,2}
For any $p,q\geq 0$, we have 
\begin{equation}\label{eq:main 3,1,2}
s^\star(p,q)=\sum_{\substack{2i+k+u=2p\\ j+l+v=q}}
(-1)^{j+k}{k+l\choose k}{u+v\choose u} 
s(i,j)\,\zeta^\star(\{2\}^{k+l})\,\zeta^\star(\{2\}^{u+v}). 
\end{equation}
\end{thm}

By substituting \eqref{eq:s(p,q)} and \eqref{eq:s^star(0,q)} 
into \eqref{eq:main 3,1,2}, we obtain an explicit formula 
for the value of $s^\star(p,q)$: 
\begin{equation*}
\frac{s^\star(p,q)}{\pi^{4p+2q}}
=\sum_{\substack{2i+k+u=2p\\ j+l+v=q}}
(-1)^{j+k}{k+l\choose k}{u+v\choose u} 
{2i+j\choose j}\frac{\beta_{k+l}\beta_{u+v}}{(2i+1)(4i+2j+1)!}, 
\end{equation*}
where 
\[\beta_r=(2^{2r}-2)\frac{(-1)^{r-1}B_{2r}}{(2r)!}. \]
In particular, when $q=0$, we can reproduce Muneta's expression 
for $s^\star(p,0)$ \cite[Theorem B]{Muneta}. 
When $q=1$, however, our result appears different from his formula 
for $s^\star(p,1)$ \cite[Theorem C]{Muneta}.  

\bigskip

In fact, our result is slightly more general than 
Theorem \ref{thm:main 3,1,2}, namely, 
the numbers $3,1,2$ are replaced by arbitrary positive integers 
$a,b,c$ such that $a+b=2c$ and $a\geq 2$. 
Moreover, it is shown as a corollary of the corresponding identity 
between finite partial sums of multiple zeta series
(see Theorem \ref{thm:main s_m}). 
In \S3, we also give an interpretation as an identity 
in the harmonic algebra. 

\section{Generating series of truncated sums}
For an integer $m\geq 0$ and an index $\bk=(k_1,\ldots,k_n)$ 
($k_1,\ldots,k_n\geq 1$), we define finite sums 
$\zeta_m(\bk)$ and $\zeta^\star_m(\bk)$ by truncating 
the series for $\zeta(\bk)$ and $\zeta^\star(\bk)$, respectively: 
\begin{equation*}
\zeta_m(\bk)=\sum_{m\geq m_1>\cdots>m_n>0}
\frac{1}{m_1^{k_1}\cdots m_n^{k_n}}, \qquad
\zeta^\star_m(\bk)=\sum_{m\geq m_1\geq\cdots\geq m_n\geq 1}
\frac{1}{m_1^{k_1}\cdots m_n^{k_n}}. 
\end{equation*}
Here an empty sum is read as $0$. 
When $n=0$, we denote by $\varnothing$ the unique index of length zero, 
and put $\zeta_m(\varnothing)=\zeta^\star_m(\varnothing)=1$ for all $m\geq 0$. 

In the following, we fix positive integers $a$, $b$ and $c$ 
satisfying $a+b=2c$. 
For integers $p,q\geq 0$, let $I_{p,q}=I_{p,q}^{a,b,c}$ denote 
the set of all indices obtained by shuffling two sequences 
$(\{a,b\}^q)$ and $(\{c\}^p)$. For example, 
\begin{gather*}
I_{0,0}=\{\varnothing\}, \qquad I_{1,1}=\{(a,b,c),(a,c,b),(c,a,b)\}, \\
I_{1,2}=\{(a,b,c,c),(a,c,b,c),(a,c,c,b),(c,a,b,c),(c,a,c,b),(c,c,a,b)\}.
\end{gather*}

Let us consider the sums of truncated MZVs and MZSVs 
analogous to $s(p,q)$ and $s^\star(p,q)$ in the introduction: 
\[s_m(p,q)=\sum_{\bk\in I_{p,q}}\zeta_m(\bk), \quad 
s_m^\star(p,q)=\sum_{\bk\in I_{p,q}}\zeta^\star_m(\bk). \]
Then Theorem \ref{thm:main 3,1,2} is obtained from the following identity 
by putting $(a,b,c)=(3,1,2)$ and letting $m\to\infty$: 

\begin{thm}\label{thm:main s_m}
For any $p,q\geq 0$ and $m\geq 0$, we have 
\begin{equation}\label{eq:main s_m}
s_m^\star(p,q)=\sum_{\substack{2i+k+u=2p\\ j+l+v=q}}
(-1)^{j+k}{k+l\choose k}{u+v\choose u} 
s_m(i,j)\,\zeta^\star_m(\{c\}^{k+l})\,\zeta^\star_m(\{c\}^{u+v}). 
\end{equation}
\end{thm}

If we put 
\begin{alignat*}{2}
F_m(x,y)&=\sum_{p,q\geq 0}s_m(p,q)x^{2p}y^q, \quad &
H_m(z)&=\sum_{r\geq 0}\zeta_m(\{c\}^r)z^r, \\
F^\star_m(x,y)&=\sum_{p,q\geq 0}s^\star_m(p,q)x^{2p}y^q, \quad &
H^\star_m(z)&=\sum_{r\geq 0}\zeta^\star_m(\{c\}^r)z^r, 
\end{alignat*}
then it is not difficult to see that Theorem \ref{thm:main s_m} 
is equivalent to the following generating series identity: 

\begin{thm}\label{thm:main gen}
\begin{equation}\label{eq:main gen}
F^\star_m(x,y)=F_m(x,-y)H^\star_m(y-x)H^\star_m(y+x). 
\end{equation}
\end{thm}

\begin{rem}
Prof.\ Kaneko pointed out that, since 
\[H^\star_m(z)=\prod_{l=1}^m\biggl(1-\frac{z}{l^c}\biggr)^{-1}
=H_m(-z)^{-1}, \]
\eqref{eq:main gen} can be written more symmetrically as 
\[\frac{F^\star_m(x,y)}{H^\star_m(x+y)}=\frac{F_m(x,-y)}{H_m(x-y)}. \]
\end{rem}

To prove the identity \eqref{eq:main gen}, 
we introduce another kind of sums and their generating series. 
We define $J_{p,q}=J_{p,q}^{a,b,c}$ as the set of 
all shuffles of $(b,\{a,b\}^q)$ and $(\{c\}^p)$, e.g.\ 
\[J_{0,0}=\{(b)\}, \qquad 
J_{1,1}=\{(b,a,b,c),(b,a,c,b),(b,c,a,b),(c,b,a,b)\}, \]
and put 
\begin{alignat*}{2}
t_m(p,q)&=\sum_{\bk\in J_{p,q}}\zeta_m(\bk),\quad &
G_m(x,y)&=\sum_{p,q\geq 0}t_m(p,q)x^{2p+1}y^q,\\
t^\star_m(p,q)&=\sum_{\bk\in J_{p,q}}\zeta^\star_m(\bk),\quad &
G^\star_m(x,y)&=\sum_{p,q\geq 0}t^\star_m(p,q)x^{2p+1}y^q. 
\end{alignat*}

\begin{lem}\label{lem:f_m,g_m,F_m,G_m}
For $m\geq 0$, we have 
\begin{align}
\label{eq:F_m,G_m}
\begin{pmatrix}F_m(x,y) \\ G_m(x,y)\end{pmatrix}
&=U_mU_{m-1}\cdots U_1\begin{pmatrix}1 \\ 0\end{pmatrix}, \\
\label{eq:F^star_m,G^star_m}
\begin{pmatrix}F^\star_m(x,y) \\ G^\star_m(x,y)\end{pmatrix}
&=V_mV_{m-1}\cdots V_1\begin{pmatrix}1 \\ 0\end{pmatrix}, 
\end{align}
where 
\begin{align*}
U_l&=\begin{pmatrix}1+\frac{y}{l^c} & \frac{x}{l^a} \\ 
\frac{x}{l^b} & 1+\frac{y}{l^c}\end{pmatrix}, \\
V_l&=
\frac{1}{\bigl(1-\frac{y-x}{l^c}\bigr)\bigl(1-\frac{y+x}{l^c}\bigr)}
\begin{pmatrix}1-\frac{y}{l^c} & \frac{x}{l^a} \\ 
\frac{x}{l^b} & 1-\frac{y}{l^c}\end{pmatrix}. 
\end{align*}
\end{lem}

\begin{proof}
For $m=0$, both \eqref{eq:F_m,G_m} and \eqref{eq:F^star_m,G^star_m} 
are obvious. 
For $m\geq 1$, we write 
\begin{align*}
F_m(x,y)
&=\sum_{p,q\geq 0}\sum_{\bk\in I_{p,q}}\zeta_m(\bk)\,x^{2p}y^q\\
&=\sum_{p,q\geq 0}\sum_{(k_1,\ldots,k_{2p+q})\in I_{p,q}}
\sum_{m\geq m_1>\cdots>m_{2p+q}\geq 1}
\frac{x^{2p}y^q}{m_1^{k_1}\cdots m_{2p+q}^{k_{2p+q}}}. 
\end{align*}
We decompose this series into three partial sums, 
each consisting of the terms such that 
(i) $m_1<m$, (ii) $m_1=m$ and $k_1=a$, or (iii) $m_1=m$ and $k_1=c$, 
respectively. 
Then we obtain the equality 
\[F_m(x,y)=F_{m-1}(x,y)+\frac{x}{m^a}G_{m-1}(x,y)+\frac{y}{m^c}F_{m-1}(x,y).\]
Similarly, we also have 
\[G_m(x,y)=G_{m-1}(x,y)+\frac{x}{m^b}F_{m-1}(x,y)+\frac{y}{m^c}G_{m-1}(x,y).\]
Combining them together, we get 
\[\begin{pmatrix}F_m(x,y)\\ G_m(x,y)\end{pmatrix}
=U_m\begin{pmatrix}F_{m-1}(x,y)\\ G_{m-1}(x,y)\end{pmatrix}, \]
and hence \eqref{eq:F_m,G_m} by induction. 

In a similar way, we can show that 
\begin{align*}
F^\star_m(x,y)&=F^\star_{m-1}(x,y)
+\frac{x}{m^a}G^\star_m(x,y)+\frac{y}{m^c}F^\star_m(x,y), \\
G^\star_m(x,y)&=G^\star_{m-1}(x,y)
+\frac{x}{m^b}F^\star_m(x,y)+\frac{y}{m^c}G^\star_m(x,y), 
\end{align*}
that is, 
\[\begin{pmatrix}1-\frac{y}{m^c} & -\frac{x}{m^a} \\
-\frac{x}{m^b} & 1-\frac{y}{m^c}\end{pmatrix}
\begin{pmatrix}F_m(x,y)\\ G_m(x,y)\end{pmatrix}
=\begin{pmatrix}F_{m-1}(x,y)\\ G_{m-1}(x,y)\end{pmatrix}. \]
Since 
\[\begin{pmatrix}1-\frac{y}{m^c} & -\frac{x}{m^a} \\
-\frac{x}{m^b} & 1-\frac{y}{m^c}\end{pmatrix}^{-1}=V_m\]
under the assumption $a+b=2c$, 
we obtain \eqref{eq:F^star_m,G^star_m} by induction. 
\end{proof}

Now it is easy to prove Theorem \ref{thm:main gen}. 
Indeed, the identities \eqref{eq:F_m,G_m} and \eqref{eq:F^star_m,G^star_m} 
imply that 
\begin{align*}
\begin{pmatrix}F^\star_m(x,y) \\ G^\star_m(x,y)\end{pmatrix}
&=\prod_{l=1}^m\Biggl\{\biggl(1-\frac{y-x}{l^c}\biggr)
\biggl(1-\frac{y+x}{l^c}\biggr)\Biggr\}^{-1}
\cdot\begin{pmatrix}F_m(x,-y) \\ G_m(x,-y)\end{pmatrix}\\
&=H^\star_m(y-x)H^\star_m(y+x)
\begin{pmatrix}F_m(x,-y) \\ G_m(x,-y)\end{pmatrix}. 
\end{align*}

\begin{rem}
In the above proof, it is also shown that 
\begin{equation}\label{eq:t_m}
t_m^\star(p,q)=\sum_{\substack{2i+k+u=2p\\ j+l+v=q}}
(-1)^{j+k}{k+l\choose k}{u+v\choose u} 
t_m(i,j)\,\zeta^\star_m(\{c\}^{k+l})\,\zeta^\star_m(\{c\}^{u+v}). 
\end{equation}
\end{rem}

\section{Identities in the harmonic algebra}
In this section, we give algebraic interpretations of 
identities \eqref{eq:main s_m} and \eqref{eq:t_m}.  
First we recall the setup of harmonic algebra 
(see \cite{IKOO} for a more general discussion). 

Let $\frH^1=\Q\langle z_k\mid k\geq 1\rangle$ be 
the free $\Q$-algebra generated by countable number of variables $z_k$ 
($k=1,2,3,\ldots$).  
The harmonic product $*$ is the $\Q$-bilinear product on $\frH^1$ defined by 
\begin{gather*}
w*1=1*w=w,\\ 
z_kw*z_lw'=z_k(w*z_lw')+z_l(z_kw*w')+z_{k+l}(w*w')
\end{gather*}
for $k,l\geq 1$ and $w,w'\in\frH^1$. 
It is known that $\frH^1$ equipped with the product $*$ 
becomes a unitary commutative $\Q$-algebra, denoted by $\frH^1_*$. 

For an integer $m\geq 0$, we define a $\Q$-linear map 
$Z_m\colon \frH^1\longrightarrow\Q$ by 
\[Z_m(1)=1,\quad Z_m(z_{k_1}\cdots z_{k_n})=\zeta_m(k_1,\ldots,k_n). \]
In fact, $Z_m$ is a $\Q$-algebra homomorphism from $\frH^1_*$ to $\Q$. 
Moreover, we define a $\Q$-linear transformation on $\frH^1$ by 
\[S(1)=1,\quad S(z_k)=z_k, \quad S(z_kz_lw)=z_kS(z_lw)+z_{k+l}S(w) \]
and put $Z^\star_m=Z_m\circ S$, so that 
\[Z^\star_m(z_{k_1}\cdots z_{k_n})=\zeta^\star_m(k_1,\ldots,k_n)\]
holds for any $k_1,\ldots,k_n\geq 1$. 

\smallskip 
Now let us put 
\[\frs_{p,q}=\sum_{(k_1,\ldots,k_{2p+q})\in I_{p,q}}
z_{k_1}\cdots z_{k_{2p+q}}, \quad 
\frt_{p,q}=\sum_{(k_1,\ldots,k_{2p+q+1})\in J_{p,q}}
z_{k_1}\cdots z_{k_{2p+q+1}}. \]
Then the fact that the identity \eqref{eq:main s_m} holds 
for \emph{all} $m\geq 0$ suggests that the identities 
\begin{align}
\label{eq:frs}
S(\frs_{p,q})&=\sum_{\substack{2i+k+u=2p\\ j+l+v=q}}
(-1)^{j+k}{k+l\choose k}{u+v\choose u} 
\frs_{i,j}*S(z_c^{k+l})*S(z_c^{u+v}), \\
\label{eq:frt}
S(\frt_{p,q})&=\sum_{\substack{2i+k+u=2p\\ j+l+v=q}}
(-1)^{j+k}{k+l\choose k}{u+v\choose u} 
\frt_{i,j}*S(z_c^{k+l})*S(z_c^{u+v}) 
\end{align}
hold in $\frH^1$. Indeed, this speculation is justified 
by the following theorem: 

\begin{thm}\label{thm:Z_m}
For $w\in\frH^1$, denote the rational sequence 
$\bigl\{Z_m(w)\bigr\}_{m\geq 0}$ by $\calZ(w)$. 
Then the resulting $\Q$-algebra homomorphism 
$\calZ\colon\frH^1_*\longrightarrow\Q^{\N}$ is injective. 
\end{thm}

If we put $\frH^1_{>0}=\bigoplus_{k\geq 1}z_k\frH^1$, 
it is obvious from the definition of $Z_m$ that 
$\frH^1_{>0}=\Ker Z_0$. 
Hence it suffices to consider the map 
\[\frH^1_{>0}\longrightarrow \Q^{\Z_{>0}};\ 
w\longmapsto \bigl\{Z_m(w)\bigr\}_{m>0}. \]
The injectivity of this map is an immediate consequence 
of the following theorem, which is obtained by specializing 
Corollary 5.6 in \cite{Br}: 

\begin{thm}\label{thm:Li}
The multiple polylogarithm functions 
\[Li_\bk(t)=\sum_{m_1>\cdots>m_n>0}\frac{t^{m_1}}{m_1^{k_1}\cdots m_n^{k_n}}
=\sum_{m>0}\bigl(\zeta_m(\bk)-\zeta_{m-1}(\bk)\bigr)t^m, \]
for $\bk=(k_1,\ldots,k_n)\in(\Z_{>0})^n$ and $n\geq 1$, 
are linearly independent over the ring $\C[t,1/t,1/(1-t)]$. 
\end{thm}

\begin{rem}
It is also possible to prove the identities \eqref{eq:frs} and \eqref{eq:frt} 
directly, by making computations similar to the proof of Proposition 4 
in \cite{IKOO}, in the matrix algebra $M_2\bigl(\frH^1_*[[x,y]]\bigr)$. 
\end{rem}

\end{document}